\documentclass[a4paper,reqno,11pt,oneside]{amsart}

\usepackage{latexsym}
\usepackage{float}
\usepackage{amsfonts,amssymb,latexsym,xspace,epsfig,graphics,color}
\usepackage{amsmath,enumerate,stmaryrd,xy,stackrel}
\usepackage[footnotesize]{caption}
\usepackage{indentfirst}
\usepackage[T1]{fontenc}
\usepackage{a4wide}
\usepackage{url}
\usepackage{tikz}
\usepackage{tkz-berge}
\usepackage{color}
\usepackage{graphpap}
\usepackage{epsfig}
\usepackage{psfrag}
\usepackage{graphicx}
\usepackage{subfigure}
\usetikzlibrary{arrows,snakes,backgrounds,trees}
\usepackage{multirow}
\usepackage{threeparttable}
\usepackage{float}
\usepackage{rotating}
\usepackage{cite}

\theoremstyle{plain}
\newtheorem{theorem}{Theorem}[section]
\newtheorem{corollary}{Corollary}[section]
\newtheorem{lemma}{Lemma}[section]

\theoremstyle{definition}
\newtheorem{definition}{Definition}[section]

\theoremstyle{remark}
\newtheorem{remark}{Remark}

\numberwithin{equation}{section}
\numberwithin{figure}{section}

\begin{document}

\title[Maki-Thompson rumor model on trees]{The Maki-Thompson rumor model\\ on infinite Cayley trees}

\date{}

\author[Valdivino V. Junior]{Valdivino V. Junior}  
\address{Valdivino V. Junior. Universidade Federal de Goias, Campus Samambaia, CEP 74001-970, Goi\^ania, GO, Brazil. E-mail: vvjunior@ufg.br}

\author[Pablo Rodriguez]{Pablo M. Rodriguez} 
\address{Pablo Rodriguez. Universidade Federal de Pernambuco, Av. Prof. Moraes Rego, 1235. Cidade Universit\'aria, CEP 50670-901, Recife, PE, Brazil. E-mail: pablo@de.ufpe.br}

\author[Adalto Speroto]{Adalto Speroto} 
\address{Adalto Speroto. Universidade de S\~ao Paulo, Caixa Postal 668, CEP 13560-970, S\~ao Carlos, SP, Brazil. E-mail: speroto@usp.br}



\begin{abstract}
In this paper we study the Maki-Thompson rumor model on infinite Cayley trees. The basic version of the model is defined by assuming that a population represented by a graph is subdivided into three classes of individuals: ignorants, spreaders and stiflers. A spreader tells the rumor to any of its (nearest) ignorant neighbors at rate one. At the same rate, a spreader becomes a stifler after a contact with other (nearest neighbor) spreaders, or stiflers. In this work we study this model on infinite Cayley trees, which is formulated as a continuous-times Markov chain, and we extend our analysis to the generalization in which each spreader ceases to propagate the rumor right after being involved in a given number of stifling experiences. We study sufficient conditions under which the rumor either becomes extinct or survives with positive probability.
\end{abstract}

\maketitle
\section{Introduction}
Currently, there exist a wide variety of mathematical models formulated to describe in a simple way the phenomenon of information transmission on a population. A wide range of these models is formed by epidemic-like processes inspired by the Daley-Kendal and the Maki-Thompson models. The Daley-Kendal model has been formulated in the mid 60's as an alternative, to describe information spreading, to the well-known susceptible-infected-recovered epidemic model, see \cite{daley_nature,kendall}. Later the Maki-Thompson model has appeared in \cite{MT} as a simplification of the Daley-Kendal model. Since both models behaves asymptotically equal the Maki-Thompson model, that we just refer as the MT-model, has been used as a basis for many generalizations. 

The MT-model assumes a homogeneously mixed population of size $N+1$ subdivided into three classes of individuals: \textit{Ignorants} (those not aware of the rumor), \textit{spreaders} (who are spreading it), and \textit{stiflers} (who know the rumor but have ceased communicating it after meeting somebody who has already heard it). The number of ignorants, spreaders and stiflers at time $t$ is denoted by $X(t)$, $Y(t)$ and $Z(t)$, respectively. Initially, $X(0) = N$, $Y(0) = 1$ and $Z(0) = 0$, and $X(t) + Y(t) + Z(t) = N + 1$ for all~$t$. The MT-model is the continuous-time Markov chain $\{(X(t), Y(t))\}_{t \geq 0}$ with transitions and corresponding rates given by

\begin{equation*}
\begin{array}{cc}
\text{transition} \quad &\text{rate} \\[0.1cm]
(-1, 1) \quad &X Y, \\[0.1cm]
(0, -1) \quad &Y (N - X).
\end{array}
\end{equation*}

This means that if the Markov chain is in state $(i,j)$ at time $t$, then the probabilities of jumping to states $(i-1,j+1)$ or $(i,j-1)$ at time $t+h$ are, respectively, $i\,j\,h + o(h)$ and $j (N - i)\,h + o(h)$, where $o(h)$ represents a function such that $\lim_{h\to 0}o(h)/h =0$. In words, the rumor spreads by directed contact of spreaders with other individuals and the two possible transitions correspond to spreader-ignorant, or spreader-spreader and spreader-stifler interactions, respectively. In the first case, the spreader tells the rumor to the ignorant, who becomes a spreader, and in the other case we have the transformation of the spreader initiating the meeting into a stifler. The second transition is what we call a stifling experience and it represents the loss of interest in propagating the rumor derived from learning that it is already known by the other individual in the meeting. 

The first results for the MT-model are related to the asymptotic behavior of the proportion of ignorants at the end of the process. We refer the reader to \cite{belen04,sudbury,watson,pearce,gani-Env2000,lebensztayn/machado/rodriguez/2011a,lebensztayn/machado/rodriguez/2011b,Lebensztayn-JMAA2015} and the references therein for an overview of existing results in this direction. Also we refer the reader to \cite[Chapter~5]{dg} for an excellent account on the subject of rumor models. We point out that all these works deal with the case of homogeneously mixed populations, which is the same to say that the population is represented by a complete graph. For the case of a population represented by another type of graph we refer the reader to \cite{raey,EAP} for rigorous results based on probabilistic methods and to \cite{arruda3,moreno-PhysA2007,zanette,zanette02,MNP-PRE2004} for approximation results based on mean-field arguments and computational simulations.   

Here we are interested in the MT-model with $k$-stifling. This is the modified version of the MT-model for which each spreader ceases to propagate the rumor right after being involved in $k$ stifling experiences, for some fixed $k\in \mathbb{N}$. The known results for this model are for the case of a homogeneously mixed population; mostly, limit theorems for the remaining proportion of ignorants as the population size goes to infinity. The case $k=1$ is, of course, the MT-model. A Law of Large Numbers (LLN) for this proportion has been stated by \cite{sudbury} and a Central Limit Theorem (CLT) by \cite{watson}. The version $k=2$ has been considered by \cite{carnal} who states a LLN. Later \cite{belen08} obtain results by mean of the deterministic version for a general $k\in \mathbb{N}$ and \cite{lebensztayn/machado/rodriguez/2011a} obtain a LLN where $k$ is assumed to be a discrete random variable. Moreover, an interesting connection between the MT-model with $k$-stifling and a system of random walks has been showed by \cite{EP}.

In this work we propose by the first time studying the MT-model with $k$-stifling on a non-complete graph. We consider the model defined on an infinite Cayley tree, also known of Bethe lattice or homogeneous tree, which is an infinite connected cycle-free graph where the vertices all have the same degree; i.e., each vertex is connected to $d$ neighbours and $d$ is called the coordination number. The interest in this type of graph is twofold; on one hand its structure allows to obtain sharp results regarding survival or extinction of the rumor. The physical relevance of these results is that this type of graph represents some mean field limit of Euclidean lattices of large dimensions. On the other hand, trees are structures that usually appear in random graph models so our model may serves as an inspiration for the formulation of more general rumor models. We study the survival or not of the rumor on this graph according to the values of $d$ and $k$. Our approach relies on a comparison of our rumor model with a suitable defined branching process.   

The paper is organized as follows. Section \ref{model} is devoted to the formulation of the model and the statements of our main results. Section \ref{sec:proofs} contains the proofs of our theorems and is subdivided into two parts. In Subsection \ref{ss:dist} we study the distribution of the number of spreaders one spreader generates which is a key quantity for our results. Also this subsection is of interest by itself because of a  connection with the Coupon Collector's Problem. Subsection \ref{ss:thm} includes the construction of a underlying branching process and the proofs of our main results.

\section{The Model and Main Results}\label{model}

In what follows we let $\mathbb{T}_d=(\mathcal{V},\mathcal{E})$ for an infinite Cayley tree of coordination number $d+1$, with $d\geq 2$. The notation is coming from Graph Theory since the same graph is known as $(d+1)$-dimensional homogeneous tree. Here $ \mathcal{V} $ stands for the set of vertices and $\mathcal{E} \subset \{\{u,v\}: u,v \in \mathcal{V}, u \neq v\}$ stands for the set of edges. We shall abuse notation by writing $\mathcal{V}=\mathbb{T}_d$, and we identify one vertex as the root and denote it by ${\bf 0}$. If $\{u,v\}\in \mathcal{E}$, we say that $u$ and $v$ are neighbors, which is denoted by $u\sim v$. The degree of a vertex $v$, denoted by $deg(v)$, is the number of its neighbors. A path in $\mathbb{T}_d$ is a finite sequence $v_0, v_1, \dots, v_n $ of distinct vertices such that $ v_i \sim v_{i+1} $ for
each $i$, and a ray in $\mathbb{T}_d$ is a path with infinite vertices starting at ${\bf 0}$. Since $\mathbb{T}_d$ is a tree, there is a unique path connecting any pair of distinct vertices $u$ and $v$. Therefore we may define the distance between them, which is denoted by $d(u,v)$, as the number of edges in such path. We point out that $\mathbb{T}_d$ is a graph with an infinite number of vertices, without cycles and such that every vertex has degree $d+1$. For each $v\in \mathcal{V}$ define $|v|:=d({\bf{0}},v)$. For $ u,v \in \mathcal{V} $, we say that $u\leq v$ if $u$ is one of the vertices of the path connecting ${\bf{0}}$ and $v$; $u<v$ if $u\leq v$ and $u\neq v$. 
We call $v$ a \textit{descendant} of $u$ if $u\leq v$ and denote by $T^u = \{v \in \mathcal{V}: u \leq v\}$ the set of descendants of $u$. On the other hand, $v$ is said to be a \textit{successor} of $u$ if $u\leq v$ and $u \sim v$. For $n\geq 1$, we denote by $\partial \mathbb{T}_{d,n}$ the set of vertices at distance $n$ from the root. That is, $\partial \mathbb{T}_{d,n}= \{v \in \mathbb{T}_d: |v|=n\}$. 

\smallskip
The MT-model with $k$-stifling on $\mathbb{T}_d$ may be defined as a continuous-time Markov process $(\eta_t^{(k)})_{t\geq 0}$ with states space $\mathcal{S}=\{-1,0,1,2,\ldots,k\}^{\mathbb{T}_d}$, i.e. at time $t$ the state of the process is some function $\eta_t: \mathbb{T}_d \longrightarrow \{-1,0,1,2,\ldots,k\}$. We assume that each vertex $x \in \mathbb{T}_d$ represents an individual, which is said to be an ignorant if $\eta(x)=-1,$ a spreader who experimented $i$ stifling experiences if $\eta(x)=i$, for $i\in\{0,1,\ldots,k-1\}$, and a stifler if $\eta(x)=k.$ Remember that ignorants are those who do not know about the rumor, spreaders are those who know about the rumor and they are transmitting it, and stiflers known about the rumor but they have stopped of propagating it. Then, if the system is in configuration $\eta \in \mathcal{S},$ the state of vertex $x$ changes according to the following transition rates

\begin{equation}\label{rates}
\begin{array}{rclc}
&\text{transition} &&\text{rate} \\[0.1cm]
-1 & \rightarrow & 0, & \hspace{.5cm}  \sum_{i=0}^{k-1}n_{i}(x,\eta),\\[.2cm]
i & \rightarrow & i+1, & \hspace{.5cm}   \sum_{i=0}^{k}n_{i}(x,\eta),
\end{array}
\end{equation}

\noindent where $$n_i(x,\eta)= \sum_{y\sim x} 1\{\eta(y)=i\}$$ 
is the number of nearest neighbors of vertex $x$ in state $i$ for the configuration $\eta$, for $i\in\{-1,0,\ldots,k\}.$ Formally, \eqref{rates} means that if the vertex $x$ is in state, say, $-1$ at time $t$ then the probability that it will be in state $0$ at time $t+h$, for $h$ small, is $\sum_{i=0}^{k-1}n_{i}(x,\eta) h + o(h)$, where $o(h)$ represents a function such that $\lim_{h\to 0} o(h)/h = 0$. Roughly speaking, the rates in \eqref{rates} represent how the changes of states of individuals depend on the states of its neighbors. While the change of state of an ignorant is influenced by its spreader neighbors, the change of state for a spreader is influenced by the number of non-ignorant neighbors. We point out that stiflers do not interact with ignorants.

We call the Markov process $(\eta_t^{(k)})_{t\geq 0}$ the Maki-Thompson rumor model with $k$-stifling on $\mathbb{T}_d$, and for the sake of simplicity we abbreviate, as before, as MT-model with $k$-stifling on $\mathbb{T}_d$, or just as MT-model when $k=1$. Since we are considering a graph with an infinite number of vertices our first task shall be to define the event of survival or extinction for the rumor process.   

\smallskip
\begin{definition}
Consider the MT-model with $k$-stifling on $\mathbb{T}_d$ with initial configuration $\eta_0$ such that $\eta_0({\bf{0}})=0$ and $\eta_0(x)=-1$ for all $x\neq {\bf 0}$. We say that there is survival of the rumor if there exist a sequence $\{(v_i,t_i)\}_{i\geq 0}$, with $(v_i,t_i)\in \mathbb{T}_d \times \mathbb{R}^+$, such that $ v_0={\bf{0}}$, $t_0=0$, $v_{i+1}$ is a successor of $v_i$, $t_i < t_{i+1}$, and $\eta_{t_i}(v_i)=0$, for all $i\geq 0.$ If there is not survival, we say that the rumor becomes extinct. We denote by $\theta(d,k)$ the survival probability and we let $\theta(d):=\theta(d,1)$.
\end{definition}

In other words, by the previous definition we have that there is survival of the rumor if we can guarantee the existence of a ray from the root of $\mathbb{T}_d$ such that all the vertices in the ray were spreaders at some time. Let us start by analysing the occurrence or not of this event for the basic MT-model. 

\smallskip
\begin{theorem}\label{thm:MTpt}
Consider the MT-model on $\mathbb{T}_d$. Then $\theta(d)>0$ if, and only if, $d\geq 3$. Moreover, 
$$\theta(d) =  1-  \sum_{i=1}^{d+1} i! \dbinom{d}{i-1}\left(\frac{\psi}{d+1}\right)^i,$$
where $\psi$ is the smallest non-negative root of the equation
$$\sum_{i=0}^{d} i! \dbinom{d}{i}\left(\frac{s}{d+1}\right)^i\left(\frac{i+1}{d+1}\right)=s.$$
\end{theorem}

\smallskip
\begin{corollary}\label{cor:probab}
$\lim_{d\to \infty} \theta(d) =1$.
\end{corollary}

Theorem \ref{thm:MTpt} gains in interest if we realize that the MT-model exhibit two different behaviors according to $d=2$ or $d\geq 3$. For $d=2$ we obtain that the rumor propagates, almost surely, only to a finite set of individuals. In the other cases, for $d\geq 3$, the rumor propagates to infinitely many individuals with positive probability. In other words, the MT-model exhibit a phase transition. We do not consider the case $d=1$, the path graph with infinite vertices, because it is trivial. In that case it is enough to note that the number of spreaders will be bounded from above by the sum of two random variables with geometric law. When the MT-model is considered on the complete graph, a quantity of interest is related to the number of stiflers at the end of the process. Note that this is the number of individuals who hear about the rumor at some time. By Theorem \ref{thm:MTpt} we known that this number is finite almost surely provided $d=2$. In what follows, we give a better characterization of that number by identifying the distribution of the quantity of stiflers at the end of the process. 

\begin{theorem}\label{theo:Sinfty}
Consider the MT-model on $\mathbb{T}_2$, and let $S_{\infty}$ be the final number of stiflers at the end of the process. Then
\begin{equation}\label{eq:thmSinfty}
\mathbb{P}(S_{\infty} = i)=\frac{1}{9i}\left\{3\, \frac{G_{i}^{(i-1)}(0)}{(i-1)!}+ 8\, \frac{G_{i}^{(i-2)}(0)}{(i-2)!} + 6\, \frac{G_{i}^{(i-3)}(0)}{(i-3)!}\right\},\;\;\; i\in \mathbb{N},
\end{equation}

where 

$$G_{i}(s):=\left(\frac{2s^2 + 4s +3}{9}\right)^i,\;\;\; s\in [-1,1],$$

and 
$$G_{i}^{(j)}(s):=\frac{d^{j}(G_i(s))}{ds^j},\;\;\; \text{ for }j\geq 0,$$
\noindent
and $G_{i}^{(j)}(s):=0$ other case. Moreover, $\mathbb{E}(S_{\infty})=18.$
\end{theorem}

Another quantity useful to measure the impact of the rumor for $d=2$ is what we call the range of spreading in the following theorem. We emphasize that according to Theorem \ref{thm:MTpt}, if $d\geq 3$ then the rumor propagates to infinitely many individuals with positive probability.

\begin{theorem}\label{theo:Rinfty}
Consider the MT-model on $\mathbb{T}_2$. Let
\begin{equation}\label{eq:range} 
R:=\max\{n\geq 1: \eta_{t}(x)=1 \text{ for some }x\in \partial \mathbb{T}_{2,n}, \text{ and } t\in \mathbb{R}^+\},
\end{equation} 
be the range of spreading. Then, for any $n\geq 0$ 
\begin{equation}
\frac{3}{9}\alpha_1(n) + \frac{4}{9}\alpha_1(n)^2 + \frac{2}{9}\alpha_1(n)^3\leq \mathbb{P}(R\leq n) \leq \frac{3}{9}\alpha_2(n) + \frac{4}{9}\alpha_2(n)^2 + \frac{2}{9}\alpha_2(n)^3,
\end{equation}
where
$$\alpha_1(n):=\frac{(13/9)\{1-(8/9)^n\}}{13/9 - (8/9)^n},\;\;\;\text{ and }\;\;\;\alpha_2(n):=\frac{(4/3)\{1-(8/9)^n\}}{4/3 - (8/9)^n}.$$

\smallskip
\noindent
Besides this,
$$6.144 \leq \mathbb{E}(R)\leq 7.448.$$
\end{theorem}

Now, let us state our result related to the MT-model with $k$-stifling with $k\geq 2$. By a standard coupling argument it is not difficult to see that $\theta(d,k)$ is non-decreasing in $k$ (indeed it is non-decreasing in $d$ as well). Therefore, by Corollary \ref{cor:probab} we have $\lim_{d\to\infty}\theta(d,k)=1$ for any $k\geq 2$. In the next Theorem we prove that for the MT-model with $k$-stifling on $\mathbb{T}_d$ with $k\geq 2$ and $d\geq 2$, differently of what happens in the MT-model (i.e., $k=1$), the only behaviour is that the rumor propagates to infinitely many individulas with positive probability. Moreover, we localize the value of the survival probability as a function of $k$ and $d$. 

\smallskip
\begin{theorem}\label{thm:MTkpt}
Consider the MT-model with $k$-stifling on $\mathbb{T}_d$ with $k\geq 2$ and $d\geq 2$. Then $\theta(d,k)>0$. Moreover, let
\begin{equation}\label{eq:S*}
\mathcal{S}^{*}(i,k):=\sum_{m_1=1}^{i} \;\sum_{m_2=m_1}^{i} \cdots \sum_{m_{k-1}=m_{k-2}}^{i}\; \prod_{\ell=1}^{k-1} m_{\ell},
\end{equation}

and 

\begin{equation}\label{eq:S}
\mathcal{S}(i,k):=\sum_{m_1=1}^{i+1} \;\sum_{m_2=m_1}^{i+1} \cdots \sum_{m_{k-1}=m_{k-2}}^{i+1}\; \prod_{\ell=1}^{k-1} m_{\ell}.
\end{equation}

Then
\begin{equation}\label{eq:f1}
\theta(d,k)=1-\sum_{i=1}^{d+1} i\,  \left(\frac{\psi}{d+1}\right)^i \frac{i!}{(d+1)^k} \dbinom{d+1}{i} \mathcal{S}^{*}(i,k)
\end{equation}

\smallskip
\noindent
where $\psi$ is the smallest non-negative root of the equation
\begin{equation}\label{eq:thetak}
 \sum_{i=0}^d
	{d\choose i}\left(\frac{s}{d+1}\right)^{i}\frac{(i+1)!}{(d+1)^k} \mathcal{S}(i,k)   =s.
\end{equation}

\end{theorem}

We point out that Theorem \ref{thm:MTkpt} is usefull for the computation of the survival probability. By fixing $k$ and $d$ one can obtain this value by mean of some (computational for higher values) calculations. Table \ref{Tab} exhibits the values of $\theta(d,k)$ for $d\in\{2,3,4,5,6,7,50\}$ and $k\in\{1,2\}$. 

\begin{table}[h!]
\begin{center}
\begin{small}
\begin{tabular}{cr@{.}lr@{.}lr@{.}lr@{.}lr@{.}lr@{.}lr@{.}lr@{.}||}
\hline
$ d $ & \multicolumn{2}{c}{2} & \multicolumn{2}{c}{3} & \multicolumn{2}{c}{4} & \multicolumn{2}{c}{5} & \multicolumn{2}{c}{6} & \multicolumn{2}{c}{7} & \multicolumn{2}{c}{50} \\\hline
$ k=1 $ & 0&000000 & 0&661289 & 0&869802 & 0&931135 & 0&957300 & 0&970887  &0&999583\\
$ k=2 $ & 0&937500 & 0&991439 & 0&997434 & 0&998936 & 0&999474 & 0&999708  &0&999999 \\
\hline
\end{tabular}
\end{small}
\caption{The behavior of $\theta(d,k)$ for $k\in\{1,2\}$ and some values of $d$.}
\label{Tab}
\end{center}
\end{table}

Since $\theta(d,k)$ is non-decreasing in both $k$ and $d$ one can see that for $k\geq 3$ we should have $\theta(d,k)\approx 1$. Indeed, for $k=3$ we have
$$\mathcal{S}^{*}(i,3)=\sum_{m_1 =1}^{i} \sum_{m_2 =m_1}^{i} m_1 m_2 = \frac{1}{24} i (i+1) (3i^2 + 7i +2),$$
and
$$\mathcal{S}(i,3):=\sum_{m_1=1}^{i+1} \;\sum_{m_2=m_1}^{i+1} m_1 m_2 =  \frac{1}{24}  (i+1)(i+2) (3i^2 + 13i +12),$$
which together with \eqref{eq:f1} and \eqref{eq:thetak} for $d=2$ implies $\theta(2,3)= 0.9964$.

Remember that the original Maki-Thompson model is obtained by considering $k=1$. Consequently, our theorems contribute with the theory complementing the existing results, proved by \cite{sudbury,watson,carnal,belen08,lebensztayn/machado/rodriguez/2011a}, to the case of infinite Cayley trees. 

\section{Proofs}\label{sec:proofs}

The main idea behind our proofs is the identification of an underlying branching process related to the rumor process. Then we apply well-known results of these processes. We shall see that the offspring distribution of such a branching process is the same as the one of the number of spreaders one spreader generates. Therefore, it is enough to study the mean of this distribution to obtain results about the survival or not of the rumor and its generating function to localize the survival probability, respectively.  We subdivide this section into two parts: in the first one, Subsection \ref{ss:dist}, we study the distribution and mean of the number of spreaders one spreader generates for the MT-model first and for the MT-model with $k$-stifling later. The second part, Subsection \ref{ss:thm}, is devoted to construct the underlying branching process whose survival is equivalent with the survival of the rumor process. Also in this subsection we prove our theorems.     

\subsection{The distribution of the number of spreaders one spreader generates}
\label{ss:dist}

Let us start with the MT-model on $\mathbb{T}_d$. 
Let $N$ be the number of spreaders generated by the initial spreader. First we are interested in the law of this discrete random variable. 

\begin{lemma}
\begin{equation}\label{eq:distN}
\mathbb{P}(N=i)=i ! \dbinom{d+1}{i}\frac{i}{(d+1)^{i+1}},\;\;\; i\in \{1,\ldots, d+1\}.
\end{equation}
\end{lemma}

\begin{proof}
It is not difficult to see that $N$ takes values in the set $\{1,\ldots,d+1\}$. This is because up to become a stifler the root can contact at most $d+1$ individuals, event which happens when the stifling experience occurs only at the $(d+2)$-th contact. In general, for any $i\in \{1,\ldots,d+1\}$, $\{N=i\}$ occurs if the first $i$ contacts are with ignorants and the $(i+1)$-th contact is a stifling experience. Note that this event has probability given by
$$\mathbb{P}(N=i)=\left\{\prod_{j=0}^{i-1}\left(\frac{d+1-j}{d+1}\right)\right\}\frac{i}{d+1},$$
which can be written as in \eqref{eq:distN}. 
\end{proof}

\begin{remark}
The previous result it is of interest by itself once one realize its connection with the Coupon Collector's Problem - a classic and well-known problem in probability theory. The problem can be stated as follows:  At each stage, a collector obtains a coupon which is equally likely to be any one of $n$ types. Assuming that the results of successive stages are independent, among other results, what is the earliest stage at which all $n$ coupons have been picked at least once? This question and many interesting generalizations have been addressed in the literature, see for example \cite{coupon,kobza,Schelling}. An alternative problem is studying the number of coupons that would be expected drawn up to seeing a duplicate; that is, a coupon that already is part of the collection.  As far as we know, no attention has been paid to this quantity before. We point out that the law of such a variable is the one given by \eqref{eq:distN} with $n=d+1$. 
\end{remark}

In what follows we consider the number of spreaders one given spreader (different from the root) generates. Let $X$ be such a number.

 \begin{lemma}\label{lem:EX}
\begin{equation}\label{eq:distX}
\mathbb{P}(X=i)=\dbinom{d}{i}\frac{(i+1)!}{(d+1)^{i+1}},\;\;\;i\in \{0,\ldots, d\}.\\[.2cm]
\end{equation}

\smallskip
\noindent
Moreover, $\mathbb{E}(X)>1$ if, and only if, $d\geq 3$.
\end{lemma}

\begin{proof}
The law \eqref{eq:distX} may be obtained by observing that $X = N-1$ in law. Now, since $\mathbb{P}(X=i)>0$ for any $i\in \{0,\ldots,d\}$ we have
$$\mathbb{E}(X)=\sum_{i=0}^d i \mathbb{P}(X=i) > P(X=2) + \sum_{i=1}^d \mathbb{P}(X=i),$$
and since $\mathbb{P}(X=2)>\mathbb{P}(X=0)$ provided $d\geq 3$ we conclude
$$\mathbb{E}(X)> \sum_{i=0}^d \mathbb{P}(X=i) =1.$$
That is, $\mathbb{E}(X)>1$ provided $d\geq 3$. For $d=2$, and after some calculations, we obtain $\mathbb{E}(X)=8/9$ so the proof is complete.
\end{proof}

In what follows we consider the MT-model with $k$-stifling, for $k\geq 2$ and for this process we denote by $N^{(k)}$  (or by $X^{(k)}$) the number of spreaders the root (or another spreader) generates. 	
	
	\begin{lemma}\label{lem:distgeral}
\begin{equation}\label{eq:distgeral}
\mathbb{P}\left(X^{(k)}=i\right)= {d\choose i}\,\frac{(i+1)!}{(d+1)^{i+k}}\, \mathcal{S}(i,k), \;\;\; i \in \{0,1,2,3,\ldots,d\}. 
\end{equation}	
Moreover, $\mathbb{E}(X^{(k)})>1$ for any $d\geq 2$.
	\end{lemma}
	
\begin{proof}
It is not difficult to see that $X^{(k)}$ takes values on the set $\{0,1,\ldots,d\}$. Note that, for any $i\in \{0,1,\ldots,d\}$, $\{X^{(k)}=i\}$ occurs if, and only if, we have exactly $k-1$ stifling experiences between the first $(i+k-1)$ contacts and we have the $k$-th stifling experience at the $(i+k)$-th contact with another individual. Let $1\leq m_1 < m_2 < \cdots < m_{k-1} \leq i+k-1$ and let $A_{m_1, m_2, \ldots, m_{k-1}}$ be the event of the $k$ stifling experiences occur at the $m_1$-th, $m_2$-th, $\ldots$, $m_{k-1}$-th, and $(i+k)$-th contacts with other individuals, respectively. Thus defined we can write
$$\{X^{(k)}=i\} = \bigcup_{1\leq m_1 < m_2 < \cdots < m_{k-1} \leq i+k-1}A_{m_1, m_2, \ldots, m_{k-1}},$$
and since
$$\mathbb{P}\left(A_{m_1, m_2, \ldots, m_{k-1}}\right) ={d\choose i}\,\frac{(i+1)!}{(d+1)^{i+k}} \prod_{\ell=1}^{k-1} (m_{\ell}-\ell+1),$$\
we get \eqref{eq:distgeral}.

In order to prove that $\mathbb{E}(X^{(k)})>1$ for any $d\geq 2$ we shall consider first the case $k=2$. By \eqref{eq:distgeral} we have that
\begin{equation}
\mathbb{P}\left(X^{(2)}=i\right)= \dbinom{d}{i} \, \frac{(i+1)\,(i+2)!}{2(d+1)^{i+1}}, \;\;\; i \in \{0,1,2,3,\ldots,d\}. 
\end{equation}
In particular, note that for any $d\geq 2$ we have
$$\mathbb{P}\left(X^{(2)}=0\right)=\frac{1}{(d+1)^2}<\frac{18d(d-1)}{(d+1)^4}=\mathbb{P}\left(X^{(2)}=2\right).$$
Thus, by applying similar arguments as in Lemma \eqref{lem:EX} we have 
$$\mathbb{E}(X^{(2)}) > \mathbb{P}(X^{(2)}=2)+1- \mathbb{P}(X^{(2)}=0)>1.$$

Since a tree is a graph without cycles, and since we are assuming that only the root is a spreader at time zero, a standard coupling argument allow us to conclude that given $d\geq 2$ we have for any $k\geq 3$ that $\mathbb{E}(X^{(k)})\geq \mathbb{E}(X^{(2)})$. This complete the proof. 
\end{proof}

\begin{remark} Coming back to the Coupon Collector's Problem, and analogously as for the case $k=1$, $X^{(k)}$ has the same distribution as the number of coupons that would be drawn up to seeing the $k$th coupon that already is part of the collection provided the collection is initially formed by one coupon.
\end{remark}

\subsection{Proof of Main Theorems}
\label{ss:thm}

\subsubsection{The underlying branching process}
\label{ss:BP}
 Consider the MT-model on $\mathbb{T}_d$ with $k$-stifling, and assume that $\eta_{0}({\bf 0})=0$ and $\eta_{0}(x)=-1$ for any $x\neq {\bf 0}$. For any $n\geq 0$ we let
$$\mathcal{B}_{n}:=\left\{v\in \partial \mathbb{T}_{d,n+1}: \bigcup_{t> 0} \{\eta_{t}(v)=1\right\},$$
and we define the random variable $Z_n:= |\mathcal{B}_n|$. Thus defined, $\mathcal{B}_{0}$ is formed by those vertices at distance one from ${\bf 0}$ which are spreaders at some time, $\mathcal{B}_{1}$ is formed by those vertices at distance two from ${\bf 0}$ which are spreaders at some time, and so on. Moreover, $Z_0$ is equal to $N^{(k)}$ in law, and it is not difficult to see that 
\begin{equation}\label{eq:BPZ}
Z_{n+1}=\sum_{i=1}^{Z_n}X_i^{(k)},
\end{equation}
where $X_1^{(k)},X_2^{(k)},\ldots$ are independent copies of $X^{(k)}$. Thus defined, $(Z_n)_{n\geq 0}$ is a branching process such that $Z_0$ has a law given by $N^{(k)}$ and the offspring distribution is given by $X^{(k)}$ (see \eqref{eq:distgeral}). For a complete reference of the Theory of Branching Processes we refer the reader to \cite{BranchingProcesses}. Our construction gains in interest if we realize the following connection between the rumor model and the branching process. 

\begin{lemma}\label{lem:BPRU}
The MT-model on $\mathbb{T}_d$ survives if, and only if, the branching process $(Z_n)_{n\geq 0}$ survives.
\end{lemma}

\begin{proof}
It is direct by construction. 
\end{proof}

\subsubsection{Proof of Theorems \ref{thm:MTpt} and \ref{thm:MTkpt}}

We shall use the well-known fact that a branching process survives with positive probability if, and only, if, the mean of the offspring distribution is greater than $1$. Moreover, the survival probability can be obtained as the smallest root in $(0,1]$ of the equation $\varphi(s)=s$, where $\varphi$ is the generating probability function of the offspring distribution. Note that for any generating probability function $\varphi$ we can guarantee $\varphi(1)=1$. Therefore, the proof of Theorem \ref{thm:MTpt} is a consequence of Lemma \ref{lem:BPRU} and Lemma \ref{lem:EX}. Indeed, the MT-model survives with probability positive if, and only if, the underlying branching process does it, which happens if, and only if, $\mathbb{E}(X)>1$. Analogously, the proof of Theorem \ref{thm:MTkpt} is a consequence of Lemma \ref{lem:BPRU} and Lemma \ref{lem:distgeral}.

\subsubsection{Proof of Corollary \ref{cor:probab}} We shall verify that
$$\lim_{d\rightarrow \infty}\sum_{i=1}^{d+1}i! \dbinom{d+1}{i} \left ( \frac{\psi}{d+1} \right)^i  \frac{i}{d+1}=0. $$ Since the limit in the left side of the previous equality can be written as  
\begin{equation}\label{eq:cor1}
\lim_{d\rightarrow \infty} \frac{\psi}{d+1} \sum_{i=1}^{d+1}i! \dbinom{d+1}{i} \left ( \frac{\psi}{d+1} \right)^{i-1}  \frac{i}{d+1},
\end{equation}
and since $|\psi|\leq 1$ so $\lim_{d\rightarrow \infty} \psi/(d+1)=0,$ it is enough if we show that the right side factor in \eqref{eq:cor1} is bounded for any $d$. It is not difficult to see that
$$\sum_{i=1}^{d+1}i! \dbinom{d+1}{i} \left ( \frac{\psi}{d+1} \right)^{i-1}  \frac{i}{d+1}<\sum_{i=1}^{\infty}i \left(\frac{d\psi}{d+1}\right)^{i-1}\leq \sum_{i=1}^{\infty}i \left(\frac{d}{d+1}\right)^{i-1}<\infty.$$
Therefore the proof is complete.

\subsubsection{Proof of Theorem \ref{theo:Sinfty}}

Consider the MT-model on $\mathbb{T}_2$, and let $S_{\infty}$ be the final number of stiflers at the end of the process. We already prove that the MT-model may be seen as the branching process given by \eqref{eq:BPZ}. Therefore $S_{\infty}$ coincides with the total progeny of such a branching processes. In order to prove Theorem \ref{theo:Sinfty} we appeal to \cite{dwass}. Indeed, notice that

\begin{equation}\label{eq:thmSinfty1}
\mathbb{P}(S_\infty = i)=\sum_{n=1}^3 \mathbb{P}(S_\infty = i|N=n)\mathbb{P}(N = n),
\end{equation}
where

\begin{equation}
\label{eq:thmSinfty2}
\mathbb{P}(N = 1)=\frac{3}{9},\;\;\; \mathbb{P}(N = 2)=\frac{4}{9},\;\;\;\mathbb{P}(N = 3)=\frac{2}{9},
\end{equation}
and by \cite[Main Theorem, p. 682]{dwass}
\begin{equation}\label{eq:thmSinfty3}
\mathbb{P}(S_\infty = i|N=n)=\left\{
\begin{array}{lc}
\displaystyle\frac{n}{i}\,\mathbb{P}\left(\sum_{j=1}^i X_j = i-n\right),& i\geq n,\\[.4cm]
0,&\text{ other wise.}
\end{array} 
\right.
\end{equation}
Here $X_1, X_2, \ldots$ denote independent and identically distributed random variables with common   law given by \eqref{eq:distX} (with $d=2$). If $G_i(s)$ and $G_{X_j}(s)$ are the probability generating functions of $\sum_{j=1}^i X_j$ and $X_j$, respectively, then we have
\begin{equation}\label{eq:thmSinfty4}
G_i(s)=\prod_{j=1}^{i} G_{X_j}(s) = \left(\frac{2s^2 + 4s +3}{9}\right)^i.
\end{equation}
By joining \eqref{eq:thmSinfty1} to \eqref{eq:thmSinfty4} we get \eqref{eq:thmSinfty}. 

Now let us prove that $\mathbb{E}(S_\infty )=18$. Since
\begin{equation}\label{eq:loco}
\mathbb{E}(S_\infty)=\sum_{n=1}^3 \mathbb{E}(S_\infty|N=n)\mathbb{P}(N = n),
\end{equation}
and $\mu:=E(X)=8/9$ (see Lemma \ref{lem:EX} for $d=2$) we have for $n\in \{1,2,3\}$
\begin{equation}\label{eq:loco2}
\mathbb{E}(S_\infty|N=n)= \left(1+ n\sum_{i=0}^{\infty}\mu^i\right)=1+9n.
\end{equation}
We conclude by \eqref{eq:thmSinfty2}, \eqref{eq:loco} and \eqref{eq:loco2} that $\mathbb{E}(S_{\infty})=18$.

\subsubsection{Proof of Theorem \ref{theo:Rinfty}}

Before proving Theorem \ref{theo:Rinfty} we state an auxiliary result regarding branching processes. 

\begin{lemma}\label{lem:BPT}
Let $(Z_{n})_{n\geq 0}$ be a branching processes with $Z_0=1$ and offspring distribution given by \eqref{eq:distX} (with $d=2$). Let $G_X$ and $G_{Z_n}$ be the probability generating functions of $X$ and $Z_n$, respectively. Then,
\begin{equation}\label{eq:FLGX}
13/45 + (128s)/\{45(5-s)\} \leq G_X(s) \leq 1/3 + (2s)/(4-s), \;\;\;s\in [-1,1],
\end{equation}

and, for $n\geq 0$ 

\begin{equation}\label{eq:Tbounds}
\frac{(13/9)\left\{1-(8/9)^n\right\}}{13/9 - (8/9)^n} \leq \mathbb{P}(T\leq n) \leq \frac{(4/3)\left\{1-(8/9)^n\right\}}{4/3 - (8/9)^n}.
\end{equation}
\end{lemma}

\begin{proof}
The spirit behind the proof of the lemma is to apply \cite[Main Theorem, p. 450]{hwang}. The first step is to check the condition
$$h(s):=G_X^{(1)}(1)G_X^{(1)}(s)(1-s)^2-(1-G_X(s))^2<0,$$
where $G_{X}^{(1)}(s):=d G_X(s) / ds$. Indeed, it is not difficult to see that
$$h(s)=\frac{8}{9}\, \frac{4(s+1)}{9} (1-s)^2 -\left\{1-\frac{(2s^2+4s +3)}{9}\right\}^2,$$
and after some calculations we obtain
$$h(s)=-\frac{32}{81}(s-1)^2\left(s^2 + \frac{23}{4}s +\frac{3}{4}s\right).$$
Thus $h(s)<0$ for all $s\in [0,1)$. Therefore, by (i) from \cite[Main Theorem, p. 450]{hwang} we have that the best upper bounding fractional linear generating function for $G_X(s)$ is given by $U(s):=1/3 + (2s)/(4-s)$. Analogously, (ii) from \cite[Main Theorem, p. 450]{hwang} implies that the best lower bounding fractional linear generating function for $G_X(s)$ is given by $L(s):=13/45 + (128s)/\{45(5-s)\}$. Therefore we get \eqref{eq:FLGX}. It is well-known, see \cite{agresti/1974}, that the inequality is preserved by compositions of the same functions. This in turns implies that for any $n\geq 0$
$$L_n(s) \leq G_{X,n}(s) \leq U_{n}(s),$$
where $L_n, G_{X,n},$ and $U_{n}$ are the $n$-th composition of $L,G_{X},$ and $U$, respectively. Moreover, since $G_{X,n}(s)=G_{Z_n}(s)$, $G_{Z_n}(0)=\mathbb{P}(T\leq n)$, and by \cite[Equation (3.1)]{agresti/1974} we have
$$L_n(0)=\frac{(13/9)\left\{1-(8/9)^n\right\}}{13/9 - (8/9)^n},\;\;\; U_n(0)=\frac{(4/3)\left\{1-(8/9)^n\right\}}{4/3 - (8/9)^n};$$
we conclude \eqref{eq:Tbounds} and the proof is complete.


\end{proof}

Consider the MT-model on $\mathbb{T}_2$ and let $R$ be given by \eqref{eq:range} the range of the spreading. Note that, for any $n$
$$\mathbb{P}(R\leq n)=\sum_{i=1}^{3}\mathbb{P}(R\leq n | N=i)\mathbb{P}(N=i)=\sum_{i=1}^{3}\mathbb{P}(T\leq n)^i\mathbb{P}(N=i),$$
where $T$ is the extinction time of a branching process $(Z_{n})_{n\geq 0}$ with $Z_0=1$ and  offspring distribution given by \eqref{eq:distX} (with $d=2$). From the law of $N$, see \eqref{eq:distN} (with $d=2$) we have 

\begin{equation}\label{eq:RT}
\mathbb{P}(R\leq n)=\frac{3}{9}\, \mathbb{P}(T\leq n) + \frac{4}{9}\, \mathbb{P}(T\leq n)^2 + \frac{2}{9} \, \mathbb{P}(T\leq n)^3.
\end{equation}

In addition, by Lemma \ref{lem:BPT} we get 
$$\alpha_1(n) \leq  \mathbb{P}(T\leq n) \leq \alpha_2(n).$$ 

In order to find a lower and upper bound for $\mathbb{E}(R)$ we use \eqref{eq:RT} to obtain
$$\mathbb{P}(R>n) = \frac{17}{9}\,\mathbb{P}(T>n) - \frac{10}{9}\,\mathbb{P}(T>n)^2 + \frac{2}{9}\,\mathbb{P}(T>n)^3.$$
Therefore

\begin{equation}\label{eq:RTlong}
\mathbb{E}(R)=\sum_{n=0}^{\infty} \mathbb{P}(R>n)=\frac{17}{9}\, \sum_{n=0}^{\infty}\mathbb{P}(T>n) - \frac{10}{9}\,\sum_{n=0}^{\infty}\mathbb{P}(T>n)^2 + \frac{2}{9}\,\sum_{n=0}^{\infty}\mathbb{P}(T>n)^3.
\end{equation}
Again, by Lemma \ref{lem:BPT} we have the following bounds for the series of the previous expression:
$$4.4619\leq \sum_{n=0}^{\infty}\frac{(1/3)(8/9)^n}{4/3-(8/9)^n} \leq \sum_{n=0}^{\infty}\mathbb{P}(T>n)\leq \sum_{n=0}^{\infty}\frac{(4/9)(8/9)^n}{13/9-(8/9)^n} \leq 4.9792,$$

$$2.0982\leq \sum_{n=0}^{\infty}\left\{\frac{(1/3)(8/9)^n}{4/3-(8/9)^n}\right\}^2 \leq \sum_{n=0}^{\infty}\mathbb{P}(T>n)^2\leq \sum_{n=0}^{\infty}\left\{\frac{(4/9)(8/9)^n}{13/9-(8/9)^n}\right\}^2 \leq 2.3592,$$

$$1.5189\leq \sum_{n=0}^{\infty}\left\{\frac{(1/3)(8/9)^n}{4/3-(8/9)^n}\right\}^3 \leq \sum_{n=0}^{\infty}\mathbb{P}(T>n)^3\leq \sum_{n=0}^{\infty}\left\{\frac{(4/9)(8/9)^n}{13/9-(8/9)^n}\right\}^3 \leq 1.6804.$$

Finally, by a suitable application of the previous bounds in \eqref{eq:RTlong} we get

$$6.144 = \frac{17}{9}\, 4.4619  - \frac{10}{9}\, 2.3592 + \frac{2}{9} 1.5189 \leq \mathbb{E}(R)\leq \frac{17}{9}\, 4.9792  - \frac{10}{9}\, 2.0982+ \frac{2}{9}\, 1.6804 = 7.448.$$

\section*{Acknowledgements}
Part of this work has been developed during a visit of V.V.J. to the ICMC-University of S\~ao Paulo and a visit of P.M.R. to the IME-Federal University of Goi\'as. The authors thank these institutions for the hospitality and support. Special thanks are also due to the two anonymous reviewers for their helpful comments and suggestions.

\end{document}